\documentclass[11pt]{article}
\usepackage[utf8]{inputenc}
\usepackage{amsmath}
\usepackage{amsfonts}
\usepackage{amssymb}
\usepackage{geometry}
\usepackage{titlesec}
\usepackage{lipsum}
\usepackage{tikz-cd}
\usepackage{amsmath}
\usepackage{parskip}
\usepackage{tikz}
\usepackage[all]{xy}
\usepackage{graphics}
\usepackage{enumitem}
\usepackage{pgfplots}
\pgfplotsset{compat=1.15}
\usepackage{mathrsfs}
\usetikzlibrary{arrows}
\usepackage{amsthm}
\usepackage{subcaption}
\usepackage{algorithm2e}
\usepackage{amssymb}
\usepackage{comment}
\usepackage{url} 
\usepackage{hyperref}
\usepackage{float}
\usepackage{amsthm}
\usepackage{subcaption}
\usepackage{comment}
\usepackage{empheq}
\usepackage{hyperref}
\usetikzlibrary{decorations.pathreplacing}
\titleformat{\section}{\large\bfseries}{\thesection}{1em}{}
\titleformat{\subsection}{\normalsize\bfseries}{\thesubsection}{1em}{}

\usepackage{amsthm}

\newtheoremstyle{runin} 
  {6pt}                  
  {0pt}                  
  {\normalfont}          
  {}                     
  {\bfseries}            
  {.}                    
  { }                    
  {}                     

\theoremstyle{runin}
\usepackage{amsthm}

\theoremstyle{plain} 
\newtheorem{theorem}{Theorem}
\newtheorem{proposition}[theorem]{Proposition}
\newtheorem{lemma}[theorem]{Lemma}
\newtheorem{corollary}[theorem]{Corollary}

\newtheorem{remark}[theorem]{Remark}

\theoremstyle{definition}
\newtheorem{definition}{Definition}[section]
\usepackage{thmtools}

\makeatletter
\renewenvironment{proof}[1][Proof]{%
  \par\pushQED{\qed}%
  \normalfont
  \trivlist
  \item[\hskip\labelsep\bfseries #1\@addpunct{:}]%
  \mbox{}\par\nobreak
}{%
  \popQED\endtrivlist\@endpefalse
}
\makeatother
\usepackage{hyperref}
\usepackage{url}
\DeclareMathOperator{\diag}{diag}
\begin{document}
\begin{center}
{\Large \textbf{Möbius functions for pseudo-Levi subgroups in finite general linear and symplectic groups}}
\vspace{0.5cm}

\textbf{Aritra Kundu}\\
aritra@cmi.ac.in\\
RPTU, Kaiserslautern\\
\end{center}

\vspace{0.5cm}

\begin{abstract}
 In this paper, we compute the Möbius function on the set of pseudo-Levi subgroups containing a fixed maximal torus of two families of finite groups, $GL_n(q)$ and $Sp_{2n}(q)$, for some natural number $n$,   a prime power $q$. In the case of $GL_n(q)$, we assume $q>n$ and for $Sp_{2n}(q)$, we assume $q>2n+1$.
 To avoid multiple $Sp$ factors in a pseudo-Levi subgroup of $Sp_{2n}(q)$,  we assume $q$ is even in the case of $Sp_{2n}(q)$. In the case of $Sp_{2n}(q)$, we will consider the tori
 corresponding to the set of bi-partitions of $n$ consisting of odd integers.  The Möbius function for the set of pseudo-Levi subgroups of a finite group of Lie type $G$ is important for an explicit evaluation of the formula expressing the decomposition of tensor products of characters proved in Theorem 2 of \cite{nam2025multiplicity}.   
\end{abstract}

\section{Introduction}
Let $\mathbf{G}$ be a connected, reductive algebraic group that is isomorphic to a closed subgroup of $GL_n(\overline{\mathbb{F}_q})$, where $n$ is a positive integer and $q$ is a prime power. Let $F :GL_n(\overline{\mathbb{F}_q})\to GL_n(\overline{\mathbb{F}_q})$ be the Frobenius map, which maps each entry $a_{ij}$ of a matrix in $GL_n(\overline{\mathbb{F}_q})$ to $a_{ij}^q$. Let $\mathbf{G}^F=\{g\in \mathbf{G}\mid F(g)=g\}$. Clearly, $\mathbf{G}^F$ is a finite subgroup of $GL_n(\mathbb{F}_q)$.\\
Let $\mathbf{T}$ be a maximal torus of $\mathbf{G}$. If $\mathbf{T}$ is $F$-stable, i.e., $F(\mathbf{T})=\mathbf{T}$, then the normaliser of $\mathbf{T}$, $N_G(\mathbf{T})$ is also $F$-stable. Therefore, the Weyl group of $\mathbf{G}$ is also $F$-stable. Let $w_1, w_2$ be two elements of the Weyl group $W$. We say $w_1, w_2$ are $F$-conjugate if $w_1=vw_2F(v)^{-1}$ for some $v$ in $\mathbf{W}$. According to Proposition 25.1 of \cite{malletesterman}, $\mathbf{G}^F$-conjugacy classes of $F$-stable maximal tori of $\mathbf{G}$ are in one-to-one correspondence with $F$-conjugacy classes of the Weyl group (unique up to group isomorphism) $W$ of $\mathbf{G}$. In this paper, we assume $F$ acts trivially on $W$. Therefore, under the previous assumption, $\mathbf{G}^F$-conjugacy classes of $F$- stable maximal tori of $\mathbf{G}$ are in one-to-one correspondence with conjugacy classes of the Weyl group of $\mathbf{G}$. A pseudo-Levi subgroup of $\mathbf{G}^F$ is the subgroup $\mathbf{G}^F_g$, where $g$ is a semisimple element of $\mathbf{G}^F$ and $\mathbf{G}^F_g$ is the centraliser of $g$ in $\mathbf{G}^F$. From Corollary 4.5 of \cite{malletesterman}, we know that any semisimple element of $\mathbf{G}$ lies in a maximal torus. As tori are abelian groups (the centraliser of an element of $T$ will contain $T$), each pseudo-Levi subgroup of $\mathbf{G}^F$ will contain a maximal torus. Let $[w]$ be a conjugacy class of $W$ and $T_w$ be a representative of the corresponding $\mathbf{G}^F$-conjugacy class of maximal tori of $\mathbf{G}^F$. Assume that $\mathcal{T}_w$ is the set of all pseudo-Levi subgroups of $\mathbf{G}^F$ containing $T_w$. We consider $\mathcal{T}_w$ as a partially ordered set with the partial order given by set inclusion. With respect to this partial ordering, we compute the Möbius function $\mu_w(\mathfrak{T},\mathfrak{T^{\prime}})$, where $\mathfrak{T}, \mathfrak{T}^{\prime}$ are in $\mathcal{T}_w$. In this paper, we drop $w$ from $\mu_w$ whenever the underlying torus or the corresponding conjugacy class of the Weyl group is clear from the context.\\
In this paper, our main results are given in Theorem \ref{theorem 2} for $GL_n(q)$ and in Theorem \ref{main thm for SP} for $Sp_{2n}(q)$. It is worth mentioning that, in this paper, we work with the set of all pseudo-Levi subgroups rather than the set up to isomorphism. I am grateful to \textbf{Prof. Gunter Malle} for suggesting the paper \cite{nam2025multiplicity} and for supporting me at every stage of formulating my intuition for the solution to this problem.
\section{The case of $GL_n(\mathbb{F}_q)$, where $q>n$}
In this section, we consider $GL_n(\mathbb{F}_q)$ for some natural number $n$ and prime power $q$.  To simplify the notation from now on, we denote $GL_n(\mathbb{F}_q)$ as $GL_n(q)$.  A partition of $n$ is a tuple of positive integers $(\lambda_1,\dots,\lambda_r)$ for some positive integer $r$ such that $\lambda_1\geq \lambda_2\geq\dots\geq \lambda_r> 0$ and $\sum_{i=1}^r\lambda_i=n$. The trivial partition of $n$ is the partition of $n$ where $r=n$ and all $\lambda_i=1$.
The Weyl group of $GL_n(q)$ is $S_n$, the full symmetric group of order $n!$. The conjugacy classes of $S_n$ are in one-to-one correspondence with the partitions of $n$ (Proposition 1.9 of \cite{choi2010symmetric}). Therefore, the set of conjugacy classes of maximal tori of $GL_n(q)$ corresponds one-to-one with the set of partitions of $n$.\\
Let $\lambda=(\lambda_1,\dots,\lambda_r)$ be a partition of $n$ and $T_{\lambda}$ be the corresponding torus of $GL_n(q)$. Hence, $T_{\lambda}$ is isomorphic to $\prod_{i=1}^r\mathbb{F}_{q^{\lambda_i}}^{\times}$. \\
\begin{lemma} \label{lemma 1}
For every positive integer $m$, there exists an injective $  \mathbb{F}_q$-linear map $\Phi_m$ from $GL_1(q^ {m})$ to $GL_m(q)$.
\end{lemma}
\begin{proof}
    Consider a basis $\{e_1,e_2,\dots ,e_m\}$ of the vector space $\mathbb{F}_{q^m}$ over $\mathbb{F}_q$.  For $\alpha\in \mathbb{F}_{q^m}$ let $T_{\alpha}$ be the map from $\mathbb{F}_{q^m}$ to $\mathbb{F}_{q^m}$ such that $T_{\alpha}(\zeta)=\alpha\zeta$ for all $\zeta\in \mathbb{F}_{q^m}$. Clearly, $T_{\alpha}$ is an invertible linear transformation, hence $T_{\alpha}\in GL_m(q)$. Also, the map $\Phi_m(\alpha)=T_{\alpha}$ is multiplicative on $GL_1(q^m)$ i.e., $\Phi_m(\alpha\beta)=\Phi_m(\alpha)\Phi_m(\beta)$, and respects the scalar multiplication over $\mathbb{F}_q$. 
    If $T_{\alpha}=T_{\beta}$ then 
    \begin{align*}
      \alpha=T_{\alpha}(1)=T_{\beta}(1)=\beta  
    \end{align*}
    Hence $\Phi_m$ is an injective $\mathbb{F}_q$ linear homomorphism.
\end{proof}
According to Lemma \ref{lemma 1}, $V\cup\{0\}$, where $V$ is the image of $\Phi_m$, is a subspace of $M_m(q)$, where $M_m(q)$ is the set of all $m\times m$ matrices over $\mathbb{F}_q$. As $\Phi_m$ is multiplicative, $V$ is a subgroup of $GL_m(q)$.
Hence, from now on we regard $GL_1(q^m)$ as a subset of $GL_m(q)$. By a similar approach, we get that  $GL_a(q^m)$ is a subgroup of $GL_{am}$. Also, by replacing $q$ with $q^m$, we get that $GL_a(q^{mr})$ is a subgroup of $GL_{ar}(q^m)$.\\

 Consider an element $\zeta$ of $GL_1(q^m)$. $\Phi_m(\zeta)$ is conjugate to the diagonal matrix $\diag (\zeta,\zeta^q,\dots,\zeta^{q^{m-1}})$ in $GL_m(\overline{\mathbb F_q})$.
 We denote the diagonal matrix  $\diag(\zeta,\zeta^q,\dots,\zeta^{q^{m-1}})$ by $C_{\zeta}$.
Let $s$ be a semisimple element of $GL_n(q)$ that lies inside $T_{\lambda}$ such that $s=\diag (\zeta_1,\dots,\zeta_r)$ where $\zeta_i\in \mathbb{F}^{\times}_{q^{\lambda_i}}$. Hence, $s$ will be conjugate to the diagonal matrix $C_s$ of size $n\times n$, such that $C_s=\diag (C_{\zeta_1},C_{\zeta_2},\dots,C_{\zeta_r})$ in $GL_n(\overline{\mathbb{F}_q})$.\\
If all the eigenvalues of $C_s$ are distinct, then the centraliser of $s$ in $G$ is $T_{\lambda}$ and if all the eigenvalues are equal, then the centraliser of $s$ is $GL_n(q)$. Therefore, studying the equality relations of eigenvalues of $C_s$ is enough to determine $C_G(s)$.
\begin{lemma}\label{lemma 2}
    Let $s$ be a semisimple element of $GL_n(q)$ which is conjugate to $\diag (\zeta,\dots \zeta^{q^{n-1}})$ for some $\zeta\in \mathbb{F}_{q^n}^{\times}$.  If $\mathbb{F}_q(\zeta)=\mathbb{F}_{q^m}$ for some $m|n$, then $C_G(s)$ is $GL_{k}(q^m)$, where $km=n$.
\end{lemma}
\begin{proof}
    As $\zeta$  is a primitive element of $\mathbb{F}_{q^m}$, each of the Frobenius conjugates of $\zeta$ appears $k$ times in the diagonal entries. Hence, the centraliser of $s$ in $GL_n(\overline{\mathbb{F}_q})$ is the group $GL_{k}(q^m)$.
\end{proof}
\begin{lemma}\label{lemma 3}
    Let $s$ be a semisimple element of $GL_n(q)$ which is conjugate to the diagonal matrix  $C_s=\diag (C_{\zeta_1},C_{\zeta_2},\dots,C_{\zeta_m})$ and $C_{GL_n(q)}(s)=\prod_{k=1}^mGL_{a_k}(q^{d_k})$.
    Assume $d_i=d_j$ and there exists a semisimple element $t$ which is conjugate to $ \diag (C_{\zeta_1},\dots,C_{\zeta_{j-1}},C_{\zeta_i},C_{\zeta_{j+1}},\dots C_{\zeta_m})$.
   Then the centraliser of $t$ in $GL_n(q)$ is $(\prod_{k=1;k\neq i,j}^{m}GL_{a_k}(q^{d_k}))\times GL_{a_i+a_j}(q^{d_i})$.
\end{lemma}

Let us define two operations on $G=\prod_{k=1}^m GL_{a_k}(q^{d_k})$.\\
$(a)$ \textbf{Merging two factors}: If $d_i=d_j$ for some $i<j$ then $G^{\prime}$, which is \\
$\prod_{k=1;k\neq i,j}^{m}GL_{a_k}(q^{d_k})\times GL_{a_i+a_j}(q^{d_i})$, is called a group obtained by \textit{merging two factors} $GL_{a_i}(q^{d_i})$ and $GL_{a_j}(q^{d_j})$.\\
$(b)$ \textbf{Exponent reduction}: $G^{\prime}$ of the form $\prod_{k=1,k\neq i }^{m} GL_{a_k}(q^{d_k})\times GL_{a_ir_i}(q^{s_i})$ where $d_i=r_is_i$ is called a group obtained by \textit{exponent reduction} of the $i$th factor of $G$.

\begin{lemma}
    Let $G=\prod_{i=1}^m GL_{a_i}(q^{d_i})$ be a pseudo-Levi subgroup containing $T_{\lambda}$. Any pseudo-Levi subgroup containing $G$ can be obtained by using the two operations: merging two factors of the same exponent and exponent reduction of multiple factors. 
\end{lemma}
\begin{proof}
    Consider a semisimple element $s$ such that $G$ is the centraliser of $s$. As $G=\prod_{i=1}^mGL_{a_i}(q^{d_i})$, $s$ is conjugate to a diagonal matrix whose eigenvalues are the Frobenius conjugates $\alpha_i^{q^{k_i-1}}$
($1 \le k_i \le d_i$) each occur with multiplicity $a_i$, where $\alpha_i$ is a
primitive element of $\mathbb{F}_{q^{d_i}}$ and $\alpha_i, \alpha_j$ are not
Frobenius conjugate for $i \ne j$.    
    Let $t$ be a semi-simple element such that $C_{GL_n(q)}(t)\supseteq G$, which implies that
    \begin{align*}
        t\in Z(G)=\prod_{k=1}^m Z(GL_{a_k}(q^{d_k}))\cong \prod_{k=1}^m GL_1(q^{d_k}).
    \end{align*}
    
So, $t$ is conjugate to the diagonal matrix whose diagonal entries are $\{\zeta^{q^{k_i}}_i|1\leq i\leq m, 1\leq k_i\leq d_i,\zeta_i\in \mathbb{F}^{\times}_{q^{d_i}}\}$ in $GL_n(\overline{\mathbb{F}_q})$. Now if $\zeta_i$ lies in a proper subfield of $\mathbb{F}_{q^{d_i}}$ , then $C_{GL_n(q)}(t)$ contains the exponent reduction of the $ith$ factor of $G$ by Lemma \ref{lemma 2}.\\
If $\zeta_i$ and $\zeta_j$, primitive elements of $\mathbb{F}_{q^{d_i}}$ and $\mathbb{F}_{q^{d_j}}$ respectively, are Frobenius conjugate to each other, then $d_i$ must be equal to $d_j$ and hence $C_{GL_n(q)}(t)$ is obtained by merging the two factors $GL_{a_i}(q^{d_i})$ and $GL_{a_j}(q^{d_j})$ by Lemma \ref{lemma 3}.\\
So, $C_{GL_n(q)}(t)$ contains the group obtained by iteratively using the two aforementioned operations on $G$. 
\end{proof}

In subsection 2.1, we compute the Möbius function on the set of all pseudo-Levi subgroups containing the torus corresponding to the trivial partition, and in subsection 2.2, we do the same for a general torus.
\subsection{The set of pseudo-Levi subgroups containing a fixed split torus}
Let $T_1$ be a representative of the conjugacy class of maximal tori corresponding to the trivial partition of $n$. We may choose $T_1$ to be the subgroup of diagonal matrices in $GL_n(q)$. We call this torus a split torus of $GL_n(q)$. Clearly, $T_1$ is isomorphic to $\prod_{i=1}^n GL_1(q)$. As exponent reduction is not possible for $T_1$, any pseudo-Levi subgroup containing $T_1$ can be obtained by iteratively merging two factors. Every factor of $T_1$ has exponent $1$, so any factor can merge with any other factor.

\begin{lemma} \label{poset bijection}
  Let $\Pi_n$ be the poset of all set partitions of $\{1,\dots,n\}$. The poset of pseudo-Levi subgroups of $GL_n(q)$ containing $T_1$ (w.r.t. inclusion) is isomorphic to $\Pi_n$, provided $q-1>n$.  
\end{lemma}
\begin{proof}
Let $PL_{T_1}$ be the set of all pseudo-Levi subgroups containing $T_1$.
We define a map $f: PL_{T_1}\to \prod_n$ such that the $ i$-th block of $f(\prod_{i=1}^rGL_{m_i}(q))$ is the collection of the indices of factors of $\prod_{i=1}^n GL_1(q)$ that merged to form $GL_{m_i}(q)$. Clearly, $f$ is a bijection. Also, both $f$ and $f^{-1}$ are order-preserving, as the pseudo-Levi subgroups become larger only by merging additional blocks. So, $PL_{T_1}$ is isomorphic (as a poset) to $\Pi_n$.

  \end{proof}
The partition into $n$ singletons requires  $n$ distinct eigenvalues in $\mathbb{F}_q^{\times}$, hence the condition  $q-1\geq n$ is required for this Lemma.

Now, as a corollary to Lemma \ref{poset bijection}, to obtain the M\"obius functions on the pseudo-Levi subgroups, it is enough to compute the M\"obius functions of the corresponding elements in $\prod_n$.
 This leads us to a special case of this subsection:
\begin{theorem}\label{theorem 3(a)}
    Let $G_1$ and $G_2$ be two pseudo-Levi subgroups such that $G_1=\prod_{i=1}^nGL_{1}(q)$,$G_2=GL_n(q)$. Then $\mu(G_1,G_2)=(-1)^{n-1}(n-1)!$.
\end{theorem}
\begin{proof}
   The $f(G_1)$ is the trivial partition and $f(G_2)$ is the set partition consisting of the set containing all elements of $\{1,\dots,n\}$.  By Example 3.10.4 of \cite{stanleyec1}, we get that $\mu(G_1,G_2)=(-1)^{n-1}(n-1)!$.
\end{proof}

\begin{theorem}
    Let $G_1$, $G_2$ be two pseudo-Levi subgroups containing $T_1$ such that $G_1=\prod_{i=1}^kGL_{a_i}(q)\subseteq G_2=\prod_{i=1}^p GL_{b_i}(q)$. Then $\mu(G_1,G_2)=(-1)^{k-p}\prod_{i=1}^p(r_i-1)!$, where $GL_{b_i}(q)$ contains the product of $r_i$ many factors of  $G_1$.

\end{theorem}
\begin{proof}
    Let $\prod_{i=1}^pG_1^{(i)}$ be a decomposition of  $G_1$ such that $G_1^{(i)}$ is the product of $r_i$ many factors of $G_1$ and contained in $GL_{b_i}(q)$. So, $\mu(G_1,G_2)=\prod_{i=1}^p\mu(G_1^{(i)},GL_{b_i}(q))$. By Theorem \ref{theorem 3(a)},  $\mu(G_1^{(i)},GL_{b_i}(q))=(-1)^{r_i-1}(r_i-1)!$. Hence, $\mu(G_1,G_2)=(-1)^{k-p}\prod_{i=1}^p(r_i-1)!$. 
\end{proof}
\subsection{The set of pseudo-Levi subgroups containing a fixed torus in a conjugacy class corresponding to non-trivial partitions}
Suppose $P$ is a poset and $x,y$ are two elements of $P$. We define $x\wedge y$ to be the maximum element of the set $\{z|z\leq x, y\}$ if the maximum exists. Similarly, $x\vee y$ is defined as the minimum element of the set $\{z|z\geq x, y\}$, whenever it exists. If for all $x$ and $y$ in $P$, $x\wedge y$ and $x\vee y$ exist, then $P$ is called a lattice.\\
Let $\lambda$ be a partition of $n$ such that $\lambda=(\lambda_1^{m_1},\lambda_2^{m_2},\dots,\lambda_t^{m_t})$, where $\lambda_1>\lambda_2>\dots>\lambda_t\geq 1$ and $m_i$ is the multiplicity of $\lambda_i$ in $\lambda$. So, $T_{\lambda}=\prod_{i=1}^t\prod_{j=1}^{m_i} GL_1(q^{\lambda_i}).$ Pseudo-Levi subgroups containing $T_{\lambda}$ are obtained by reducing exponents and merging two factors of the same exponent.\\
Let $G_1$ and $G_2$ be two pseudo-Levi subgroups, containing $T_{\lambda}$ such that $G_1\subseteq G_2$. Let us assume $G_1=\prod_{i=1}^k GL_{a_i}(q^{d_i})$ and $G_2=\prod_{j=1}^p GL_{b_j}(q^{c_j})$ where $k,p\leq \sum _{i=1}^tm_i$. It is clear that the poset $[G_1, G_2]$, which contains all the pseudo-Levi subgroups of $G_2$ that contain $G_1$ as a subgroup, is a lattice. 
\begin{lemma}\label{lemma for lambdas}
    Let $G$ be a pseudo-Levi subgroup that is the centraliser of an element of $T_{\lambda}$, such that $G=\prod_{i=1}^rGL_{x_i}(q^{y_i})$. Then $\sum_{i=1}^rx_iy_i=n$.
\end{lemma}

\begin{theorem}(Weisner's Theorem, Corollary 3.9.3 of \cite{stanleyec1})\label{Weisner's thm}
   Let $L$ be a finite lattice with at least two elements, and let $a$ be not the highest element of $L$. Then $\sum_{x:x\wedge a=0}\mu(x,1)=0,$ where $0$ is the least element and $1$ is the highest element of $L$.  
\end{theorem}
According to Weisner's theorem, $\mu(0,1)=-\sum_{x\neq 0,x\wedge a=0}\mu(x,1)$. We use this result several times in the computations.
\begin{lemma}\label{lemma 9(a)}
    Let $G_1=\prod_{i=1}^k GL_{a_i}(q^{d_i})\subseteq G_2=\prod_{i=1}^p GL_{b_i}(q^{c_i})$ be two pseudo-Levi subgroups containing $T_{\lambda}$.
    If $GL_{a_i}(q^{d_i})\subseteq GL_{b_j}(q^{c_j})$  then $c_j\leq d_i$ and $a_id_i\leq b_jc_j$.
\end{lemma}
\begin{proof}
 By our assumption,   $GL_{a_i}(q^{d_i})\subseteq GL_{b_j}(q^{c_j})\subseteq GL_n(q)$. Hence, $Z(GL_{b_j}(q^{c_j}))\subseteq Z(GL_{a_i}(q^{d_i}))$. Now $Z(GL_{b_j}(q^{c_j}))$ and $ Z(GL_{a_i}(q^{d_i}))$ are cyclic subgroups of order $q^{c_j}-1$ and $q^{d_j}-1$ respectively. Hence, 
 \begin{align*}
     q^{c_j}-1|q^{d_i}-1
     \implies  c_j\leq d_i
 \end{align*}
 By considering the $\mathbb{F}_q$ dimensions of the underlying spaces we get $a_id_i\leq b_jc_j$. 
\end{proof}

\begin{lemma} \label{lemma p<k}
    If $G_1=\prod_{i=1}^k GL_{a_i}(q^{d_i})\subseteq G_2=\prod_{i=1}^p GL_{b_i}(q^{c_i})$, then $p\leq k$, i.e., $G_1$ has more than or the same number of factors as $G_2$.
\end{lemma} 
\begin{proof}
    For each factor $GL_{a_i}(q^{d_i})$ of $G_1$ , there exists a factor of $G_2$, $GL_{b_j}(q^{c_j})$ , which contains that factor of $G_1$.\\
    If $c_j<d_i$ and $GL_{a_i}(q^{d_i})\subseteq GL_{b_j}(q^{c_j})$, then  $a_id_i\leq b_jc_j.$ [By Lemma \\\ref{lemma 9(a)}]
   
    If $c_j=d_i$ and $GL_{a_i}(q^{d_i})\subseteq GL_{b_j}(q^{c_j})$, then $b_j\geq a_i$ which implies $ a_id_i\leq b_jc_j$.\\
    
    From Lemma \ref{lemma for lambdas}, we know that $\sum_{i=1}^k a_id_i=\sum_{j=1}^p b_jc_j=n$ and for each $i$ there exists $j$ such that $b_jc_j\geq a_id_i$. Therefore, the number of factors in $G_2$ has to be smaller than or equal to that of $G_1$. Hence, $p\leq k$.
\end{proof}
\begin{proposition}\label{prop 10}
    Let $G_1$ and $G_2$ be two pseudo-Levi subgroups containing $T_\lambda$ such that $G_1\subseteq G_2$ and $G_1=\prod_{i=1}^kGL_{a_i}(q^{d_i})$ and $G_2=\prod_{j=1}^p GL_{b_j}(q^{c_j})$. If $k=p$, then $\mu(G_1, G_2)=\prod_{i=1}^k\mu_{P}(c_i,d_i)$ where $P$ is the poset of positive integers with respect to divisibility and $\mu_P$ is the  Möbius function of $P$. 
\end{proposition}

\begin{proof}
   As $G_1\subseteq G_2$, from Lemma \ref{lemma p<k} we know that $p\leq k$. If, $p=k$ that means for all factors $GL_{a_i}(q^{d_i})$ of $G_1$,  there exists a unique factor of $G_2$, $GL_{b_j}(q^{c_j})$ such that $a_id_i=b_jc_j$ and $GL_{a_i}(q^{d_i})\subseteq GL_{b_j}(q^{c_j})$. So, by re-numbering the factors of $G_2$ without loss of generality, we assume that $GL_{a_i}(q^{d_i})\subseteq GL_{b_i}(q^{c_i})$.
   Also, note that both $GL_{a_i}(q^{d_i})$ and $GL_{b_i}(q^{c_i})$ are subgroups of $GL_{a_id_i}(q)$. Therefore, $\mu(G_1,G_2)=\prod_{i=1}^k\mu(GL_{a_i}(q^{d_i}), GL_{b_i}(q^{c_i}))$.\\
   From the proof of Lemma \ref{lemma p<k}, $a_id_i=b_ic_i$ implies $c_i\leq d_i$. If $c_i$ and $d_i$ are the same, then $a_i$ and $b_i$ are equal, so $GL_{a_i}(q^{d_i})=GL_{b_i}(q^{c_i})$ and $\mu(GL_{a_i}(q^{d_i}),GL_{b_i}(q^{c_i}))=1=\mu_P(c_i,d_i)$.\\
   Consider the situation $c_i<d_i$, i.e., $c_i|d_i$. Each element on the interval $[GL_{a_i}(q^{d_i}),GL_{b_i}(q^{c_i})]$ will be of the form $GL_{ma_i}(q^{d_i/m})$ where $c_i|(d_i/m)$. Therefore, elements inside the interval can be associated with the factors of $(d_i/c_i)$, and if two groups have an inclusion relation with each other, then their exponents are divisors of one another. Therefore, $[GL_{a_i}(q^{d_i}),GL_{b_i}(q^{c_i})]$ is isomorphic to $[c_i,d_i]$ with respect to divisibility.
\end{proof}
We have $\mu_P(c,d)\neq 0$ iff $d/c$ is square-free. If $d/c$ is square-free, then $\mu_P(c,d)=(-1)^{w(d/c)}$, where $w(d/c)$ is the number of distinct prime factors of $d/c$ (Section 3.8 of \cite{stanleyec1}).
Now let us focus on the general situation where the number of factors of $G_2$ is less than or equal to that of $G_1$. Let us assume $G_1=\prod_{i=1}^k GL_{a_i}(q^{d_i}) \subseteq G_2=\prod_{i=1}^p GL_{b_i}(q^{c_i})$ and $p\leq k$. So, according to Lemma \ref{lemma p<k}, all elements in the interval $[G_1, G_2]$ have at most $k$ many factors and at least $p$ many factors. In the next proposition, we will show that there exists a unique maximal element (with respect to inclusion) which has $k$ many factors in $[G_1, G_2]$. Before diving into the proposition, we state the notion of exponent of a group and an intuitive fact as a lemma.
\begin{definition}
    For the purpose of this paper, we define the \textit{exponent} of $GL_n(q^m)$ to be the integer $m$. Let $G$ be the product $\prod_{i=1}^sGL_{n_i}(q^{m_i})$. Then the \textit{set of exponents} of $G$ is defined to be $\{m_1,\dots,m_s\}$. 
\end{definition}
\begin{definition}
    Let $G=\prod_{i=1}^sGL_{n_i}(q^{m_i})$ and the set of exponents of $G$ is $\{m_1,\dots,m_s\}$. The \textit{multiplicity} of an element in the set of exponents $m_i$ is the number of factors present in $G$ whose exponent is $m_i$. 
\end{definition}
\begin{lemma} \label{lemma 8}
    Let $G$ be a maximal element in $[G_1, G_2]$ among all the elements that have the same number of factors as $G_1$. Then the set of exponents of $G$ is the same as the one of $G_2$.
\end{lemma}
\begin{proof}
    Let $G=\prod_{i=1}^kGL_{e_i}(q^{f_i})$. As $G\subseteq G_2$, for all $GL_{e_i}(q^{f_i})$ there exists a factor of $G_2$, $GL_{b_j}(q^{c_j})$ that contains $GL_{e_i}(q^{f_i})$. Therefore, $c_j$ is a factor of $f_i$, and let $m_i=f_i/c_j$. Let $G^{*}=\prod_{i=1}^k GL_{e_im_i}(q^{f_i/m_i})$.
    Clearly, $G\subseteq G^{*}$ and $ G^{*}\in[G_1,G_2]$. As $G^{*}$ has also $k$ many factors, $G=G^{*}$. That implies $m_i=1$ and eventually, $\{f_1,\dots,f_k\}=\{c_1,\dots,c_p\}$.
\end{proof}
\begin{proposition}\label{prop-2.5}
    Let $G_1$ and $G_2$ be two pseudo-Levi subgroups containing $T_{\lambda}$ such that $G_1=\prod_{i=1}^k GL_{a_i}(q^{d_i})\subseteq G_2=\prod_{i=1}^p GL_{b_i}(q^{c_i})$ where $p\leq k$. There exists a unique maximal element $G_3$ in $[G_1,G_2]$ with $k$ factors.
\end{proposition}
     
\begin{proof}
    Let $G_3$ and $G^{\prime}_3$ be two elements in $[G_1,G_2]$ that have $k$ many factors and both are maximal. Let $G_3= \prod_{i=1}^k GL_{e_i}(q^{f_i})$ and $G_3^{\prime}=\prod_{i=1}^k GL_{e^{\prime}_i}(q^{f^{\prime}_i})$, where $\{f_1,\dots,f_k\}=\{f_1^{\prime},\dots,f_k^{\prime}\}=\{c_1,\dots,c_p\}$ which follows from Lemma $\ref{lemma 8}$. Now, $GL_{a_1}(q^{d_1})$ is isomorphic to a subgroup of $ GL_{b_1}(q^{c_1})$ and both $GL_{e_1}(q^{f_1})$ and $GL_{e^{\prime}_1}(q^{f^{\prime}_1})$ contain $GL_{a_1}(q^{d_1})$. As both $G_3$ and $G^{\prime}_3$ are maximal, so $f_1=f^{\prime}_1=c_1$. Now, $GL_{a_1}(q^{d_1})\subseteq GL_{e_1}(q^{f_1})$, which implies $Z(GL_{e_1}(q^{f_1}))\subseteq Z(GL_{a_1}(q^{d_1}))$. As $Z(GL_{a_1}(q^{d_1}))$ is a cyclic group of order $q^{d_1}-1$, it has a unique subgroup of order $q^{c_1}-1$. Hence, $GL_{e_1}(q^{c_1})$ is isomorphic to $GL_{e^{\prime}_1}(q^{c_1})$ and which implies $e_1=e_1^{\prime}$. In this way, by deleting each factor of $G_1$ and considering the next, we can show that $G_3=G^{\prime}_3$.
\end{proof}
\begin{definition}
    The element $G_3$ in $[G_1,G_2]$ which is the largest subgroup that has the same number of factors as $G_1$ is called the \textit{Big element} of $[G_1,G_2]$. According to Lemma \ref{lemma 8}, the set of exponents of the Big element is the same as for $G_2$.
\end{definition}

\begin{lemma} \label{lemma 10}
    Let $G$ be an element in $[G_1,G_2]$ such that $G\wedge G_3=G_1$. Then $G$ can be formed by merging factors of equal exponent of $G_1$.
\end{lemma}
\begin{proof}
    Assume that $G=\prod_{j=1}^tGL_{r_j}(q^{s_j})$ and $G$ is different from $G_1$. As $G_3$ is the Big element in $[G_1,G_2]$ and $G$ is not a subgroup of $G_3$ (otherwise $G\wedge G_3=G$), the number of factors of $G$ must be strictly less than that of $G_1$. Let us assume the $j$th factor $GL_{r_j}(q^{s_j})$ of $G$ contains the product of $k_j$ many factors of $G_1$. Let us denote the product of these $k_j$ factors of $G_1$ by $G_{1,j}$, which is contained in $GL_{r_j}(q^{s_j})$. Let $H_j$ be the Big element of $[G_{1,j}, GL_{r_j}(q^{s_j})]$. Note that the set of exponents of $H_j$ is $\{s_j\}$ and $H_j$ has $k_j$ many factors. Consider the group $H$ that is the product $\prod_{j=1}^tH_j$. Therefore, $H$ has $k$ many factors and according to Proposition \ref{prop-2.5}, $H\leq G_3$. Hence, $G\wedge G_3=G_1$ implies $H=G_1$. Hence, each factor of $G$ is formed by merging factors of $G_1$ that have the same exponent.
\end{proof}

\begin{definition}
Let $G=\prod_{i=1}^sGL_{n_i}(q^{m_i})$, we say $G$ is \textit{distinct} if $m_i\neq m_{j}$ for all $i\neq j$ 
in $[s]$; that is, any two exponents of distinct factors are distinct.
\end{definition}
\begin{proposition} \label{prop-11}
    Let $G_1$ be distinct and $G_2$ have fewer factors than $G_1$ such that $G_1\subseteq G_2$, then $\mu(G_1,G_2)=0$.
\end{proposition}
\begin{proof}
According to Proposition \ref{prop-2.5}, there exists a Big element $G_3\in [G_1,G_2]$ that has the same number of factors as $G_1$. Clearly, $G_3\neq G_2$ as $G_2$ has fewer factors than $G_1$. Therefore, according to Weisner's theorem (Theorem \ref{Weisner's thm}), $\sum_{G:G\wedge G_3=G_1}\mu(G,G_2)=0$. As $G_1$ is distinct, any two factors of $G_1$ cannot be joined to form a bigger element. So, according to Lemma \ref{lemma 10}, $G\wedge G_3=G_1$ implies that $G=G_1$. Hence, in the equation obtained from Weisner's theorem, only $G_1$ appears as a summand. Therefore, $\mu(G_1,G_2)=0$. 
\end{proof}

Let the $m_i$ factors in $G_1$ have the exponent $d_i$, and there are $l$ distinct exponents present in $G_1$. Without loss of generality, the distinct exponents are denoted by $\{d_1,\dots,d_l\}$. Consider the element $H=\prod_{i=1}^lGL_{M_i}(q^{d_i}),$ where $M_i=\sum_{t:d_t=d_i}a_t$. Let $B$ be the poset of all elements in $[G_1,GL_n(q)]$ that can be formed by merging factors of the same exponents as $G_1$.
Clearly, $H$ is the maximum element in $B$ with the fewest factors among all the elements in $B$. Also, observe that $H$ is distinct. We call $H$ the \textit{maximum distinct element} of $G_1$.

\begin{theorem} \label{thm for p<k is 0}
   Let us consider the assumptions of Proposition \ref{prop-2.5}.  Let $B$ be the poset containing all the elements formed by merging factors of equal exponents of $G_1$, and let $H$ be its maximum element. If $\mu(H_1,G_2)=0$, where $H_1=H\wedge G_2$, then for all elements $G$ in $B_1=B\cap[G_1,G_2]$,  $\mu(G,G_2)=0$.
\end{theorem}
\begin{proof}

Let us define the notion of sub-rank on the set $B_1$. We define the sub-rank of $H_1$ to be $1$. An element $G$ in $B_1$ has sub-rank $i$ if the maximal length of a maximal chain of $[G,H_1]\cap B_1$ is $i$. 
 Clearly, the claim is true for $H_1$, i.e., for all elements of sub-rank $1$. Let $G$ be an element of $B_1$ that has sub-rank $i$. By the induction hypothesis, we assume that the claim holds for all the elements of sub-rank strictly less than $i$.\\
 According to Proposition \ref{prop-2.5}, there exists a Big element $G^{*}$ of $[G,G_2]$. As $G$ has more factors than $H_1$ (as in $B_1$, bigger elements can be formed only by merging two factors, no exponent reduction is possible in $B_1$) and eventually more than $G_2$, $G^{*}$ cannot be the same as $G_2$. So, by Weisner's theorem, $\mu(G,G_2)=-\sum_{G^{\prime}\wedge G^{*}=G;G^{\prime}\neq G}\mu(G^{\prime},G_2)$. Now, applying the Lemma \ref{lemma 10} for $G$, we get $G^{\prime}\wedge G^{*}=G\implies G^{\prime}\in B_1$. As $G\leq G^{\prime}$, any maximal chain of $[G^{\prime},H_1]$ can be extended to a chain of $[G,H_1]$ by adding $G$ at the beginning. So, all $G^{\prime}$ satisfying, $G^{\prime}\wedge G^{*}=G$ have sub-rank strictly less than $i$. Hence, by induction hypothesis, $\sum_{G^{\prime}\wedge G^{*}=G}\mu(G^{\prime},G_2)=0=\mu(G,G_2)$ and since $G_1\in B_1,$ in particular, $\mu(G_1,G_2)=0$.
\end{proof}
\begin{lemma}\label{lemma 16}
    Consider the notation of Theorem \ref{thm for p<k is 0}. If $\mu(H_1,G_2)\neq 0$, then the number of factors of $H_1$ and that of $G_2$ must be the same.
\end{lemma}
\begin{proof}
 Assume the contrary. Consider the big element $H_3$ in $[H_1,G_2]$. Since $G_2$ has fewer factors than $H_1$, $G_2\neq H_3$. From Lemma \ref{lemma 10}, we know that if $G\wedge H_3=H_1$, then $G$ is obtained by merging factors of equal exponents of $H_1$. Hence, $G\subseteq H$. But as $H_1=H\wedge G_2$, $G$ has to be the same as $H_1$. Therefore, according to Weisner's theorem, $0=\sum_{G:G\wedge H_3=H_1}\mu(G,G_2)=\mu(H_1,G_2)$. 
\end{proof}

According to Theorem \ref{thm for p<k is 0} if $\mu(H_1,G_2)=0$ then $\mu(G_1,G_2)=0$.\\
Now let us dig into the case when $\mu(H_1, G_2)$ is non-zero. Lemma \ref{lemma 16} implies that the number of factors of $H_1$ must be the same as the number of factors of $G_2$. 
Let $H_1=\prod_{i=1}^rH_1^{(i)}$ and $G_2=\prod_{i=1}^rG_2^{(i)}$ such that $H_1^{(i)}\subseteq G_2^{(i)}$, where $H_1^{(i)}$ is a factor of $H_1$ and $G_2^{(i)}$ is a factor of $G_2$. As $\mu(H_1,G_2)=\prod_{i=1}^r\mu(H_1^{(i)},G_2^{(i)})$, each $\mu(H_1^{(i)},G_2^{(i)})$ has to be non-zero. Let $G_2^{(i)}=GL_{b_i}(q^{c_i})$. Hence $H_1^{(i)}$ will be  $GL_{s_i}(q^{d_i}),$ where $d_i/c_i$ is square free and $s_id_i=b_ic_i$.\\

Also, $G_1=\prod_{j=1}^k GL_{a_j}(q^{d_j})\subseteq H_1$ hence there exists a decomposition of $G_1$ such that $G_1^{(i)}=\prod_{j=r_{i-1}+1}^{r_i}GL_{a_j}(q^{d_i})\subseteq H_1^{(i)}$ and $G_1=\prod_{i=1}^rG_1^{(i)}$. So, $\mu(G_1,G_2)=\prod_{i=1}^r \mu(G_1^{(i)},G_2^{(i)})$.

\begin{theorem} \label{theorem 2}
    Let $G^{(i)}_1=\prod_{j=r_{i-1}+1}^{r_i}GL_{a_j}(q^{d_i})$ and $G^{(i)}_2=GL_{b_i}(q^{c_i})$ where $\nu_i=d_i/c_i$ and $(\sum_{j=r_{i-1}+1}^{r_i} a_j)d_i=b_ic_i$. Then $\mu(G^{(i)}_1,G^{(i)}_2)=(-1)^{r_i-r_{i-1}-1}\mu_P(c_i,d_i)(r_i-r_{i-1}-1)!.$
\end{theorem}
\begin{proof}
To avoid cumbersome notation depending on $i$, we set $S_1:=G_1^{(i)}$, $S_2:=G_2^{(i)}$, and $r_i-r_{i-1}=:r$.
   We will use induction on the number of factors present in $S_1$. If $S_1$ has only one factor, clearly, $\mu(S_1,S_2)=\mu_P(c_i,d_i)$ from Proposition \ref{prop 10}. According to induction hypothesis, for all groups $S_1^{\prime}$ that have $r-1$ factors, $\mu(S_1^{\prime},S_2)=(-1)^{r-2}\mu_P(c_i,d_i)(r-2)!$.
Let $S^{*}=GL_{\nu_i(a_{r_{i-1}+1}+\dots+a_{{r_i}-1})}(q^{c_i})\times GL_{\nu_ia_{r_i}}(q^{c_i})$ which is clearly not the same as $S_2$, where $\nu_i=d_i/c_i$. If $S\wedge S^{*}=S_1$, $S$ can either be $S_1$ or $U_k=\prod_{j=r_{i-1}+1;j\neq k}^{r_i-1}GL_{a_j}(q^{d_i})\times GL_{a_{k}+a_{r_i}}(q^{d_i})$, for some $k=r_{i-1}+1,\dots,(r_i-1)$.\\
Therefore, according to Weisner's theorem, $\mu(S_1, S_2)+\sum_{k=r_{i-1}+1}^{r_i-1}\mu(U_k,S_2)=0$. From our induction hypothesis, we know, $\mu(U_j,S_2)=(-1)^{r-2}\mu_P(c_i,d_i)(r-2)!$. Hence, 
\begin{align*}
   \mu(S_1,S_2)=-\sum_{j=r_{i-1}+1}^{r_i-1} (-1)^{r-2}\mu_P(c_i,d_i)(r-2)!=-(r-1)(-1)^{r-2}\mu_P(c_i,d_i)(r-2)!=\\
   (-1)^{r-1}\mu_P(c_i,d_i)(r-1)!. 
\end{align*}

\end{proof}

From Theorem \ref{theorem 2} we can determine $\mu(G_1^{(i)},G_2^{(i)})$ and taking the product over $i$ we get $\mu(G_1,G_2)$.

 \subsection{Remark for designing an algorithm to compute $\mu(G_1,G_2)$}

Given $G_1$ and $G_2$, first we need to compute the maximum distinct element $H$ by combining the factors of $G_1$ that have the same exponent. In the second step, we compute $H_1$ by computing the meet of $H$ and $G_2$. If $H_1$ and $G_2$ have different numbers of factors, then  $\mu(G_1, G_2)$ is $0$. Otherwise, by studying the exponents of the factors of $H_1$ and $G_2$, we can determine the pair $(H^{(i)}_1, G_2^{(i)})$ such that $H_1^{(i)}\subseteq G_2^{(i)}$ in time complexity of order of the number of factors of $G_2$. The most time-consuming step is to decompose $G_1$ in $\prod G^{(i)}_1$, where $G_1^{(i)}$ are products of certain factors of $G_1$ such that the product of the corresponding $G_1$ factors equals $H_1^{(i)}$. As the set of exponents of $G_1$ and $H_1$ is equal, we need to partition the factors of the same exponent of $G_1$ so that, by combining the factors present in a partition, we can match the dimension of $H_1^{(i)}$. After getting this decomposition, we can compute each $\mu(G_1^{(i)},G_2^{(i)})$ by Theorem \ref{theorem 2}  and determine $\mu(G_1,G_2)$.

\newpage

\section{The case of $Sp_{2n}(q)$, when $q$ is even and $q>2n+1$}
Here throughout we consider the symplectic group $Sp_{2n}(q)$ for even $q$.
\begin{definition}
    We say a tuple $(\lambda,\rho)$ is a bi-partition of $n$ if $\lambda$ and $\rho$ are partitions of $n_1$ and $n_2$ respectively and $n_1+n_2=n$.
\end{definition}
Let $W$ be the Weyl group of $Sp_{2n}(q)$. From Section 2.1 of \cite{luebeckmalle2016}, we know that the set of conjugacy classes of $W$ has a one-to-one correspondence with the set of all bi-partitions $(\lambda,\rho)$ of $n$.\\
Let $T_{\lambda,\rho}$ be the corresponding torus. In this paper, we will deal with the case where $\lambda$ and $\rho$ consist of odd integers.
We know that $T_{\lambda,\rho}$ is the product of cyclic groups $\prod_{i=1}^r C_{q^{\lambda_i}-1}\times \prod_{j=1}^s C_{q^{\rho_j}+1}$, where $C_n$ is the cyclic group of order $n$ [Section 2.1 \cite{luebeckmalle2016}]. Assume $s$ is a semi-simple element of $T_{\lambda,\rho}$ whose characteristic polynomial is $(x-1)^m\times \prod_{i=1}^k f_i^{m_i} \times \prod_{i=1}^p f^{\prime}_i$, where $f_i$ is linear and $f^{\prime}_i$ is unitary. Then the centraliser of $s$,$G_s$ is $Sp_{2m}\times \prod_{i=1}^k GL_{m_i}(q^{d_i})\times \prod_{i=1}^p GU_{m^{\prime}_i}(q^{d^{\prime}_i}),$ where $d_i$ is the degree of $f_i$ and $d^{\prime}_i$ is the degree of $f_i^{\prime}$[Section 1.10 \cite{FS89}]. The general unitary group $GU_n(q)$ is the isometry group of a non-degenerate hermitian form on $(\mathbb{F}_{q^2})^n$, i.e., $GU_n(q)=\{g\in GL_n(q^2)|f(gv,gw)=f(v,w)\forall v,w\in (\mathbb{F}_{q^2})^n\}$ for some hermitian form $f$. 
Before going into further details, let us state a simple result of number theory which will help to understand the necessity of the condition that $\lambda$ is a partition consisting of odd integers. 
\begin{lemma}\label{gcd}
    The integers $q^i-1$ and $q^j+1$ are coprime, where $q$ is even and $i$ is odd.
\end{lemma}
\begin{proof}
    Let $\ell$ be a common prime divisor of $q^i-1$ and $q^j+1$. As both numbers are odd, $\ell$ must be an odd prime.  Assume $t$ is the order of $q$ in the multiplicative group $\mathbb{F}_\ell$. So, $t$ has to divide $i$ as $q^i\equiv 1 $ mod $(\ell)$ , which means $t$ is odd. Also, $q^{2j}\equiv 1$ mod $(\ell)$ ; hence, $t|2j$ and $t$ is not a factor of $j$ as $q^j\equiv -1$ mod $(\ell)$. So, $t$ has to be even. This leads to a contradiction. 
 
\end{proof}
\begin{corollary}\label{cr}
    The Lemma \ref{gcd} implies that if $\lambda$ is a partition containing odd integers and $q$ is even, then the cyclic groups of  $q^{\lambda_i}-1 $-th roots of unity and  $q^{\rho_j}+1$-th roots of unity intersect trivially. 
\end{corollary}

\begin{lemma}
    Let $T_{\lambda,\rho}$ be a maximal torus of $Sp_{2n}(q)$ corresponding to a bi-partition of $n$, $(\lambda,\rho)$ where $\lambda$ consists of odd integers and $q$ is even. Let $G_s$ be a pseudo-Levi subgroup containing $T_{\lambda,\rho}$and $G_t$ be another pseudo-Levi subgroup containing $G_s$. Then every $GL$  and $GU$ factors of $G_t$ only contain $GL$  and $GU$ factors of $G_s$ respectively.
\end{lemma}
\begin{proof}
    Let $t$ be a semi-simple element of $T_{\lambda,\rho}$ such that $G_s= Sp_{2m}\times \prod_{i=1}^k GL_{m_i}(q^{d_i})\times \prod_{i=1}^p GU_{m^{\prime}_i}(q^{d^{\prime}_i})\subseteq G_t$. So, $t\in Z(G_s)$. Therefore, the eigenvalues of $t$ are the elements of the tuple $(1^m,s_1,\dots s_k,s^{\prime}_1,\dots s^{\prime}_p)$ where $s_i\in C_{{q^{\lambda_i}}-1}$ and $s^{\prime}_j\in C_{q^{\rho_j}+1}$. The number of $1$ that occur as eigenvalues of $t$ is the same as the rank of the $Sp$ factor of $G_t$. The characteristic polynomial of $s_i$ is linear, and the characteristic polynomial of $s_i^{\prime}$ is unitary. Now, from Corollary \ref{cr}, if $s_i=s^{\prime}_j$ then both have to be equal to $1$. So, the $GL$ factors and $GU$ factors of $G_s$ cannot fuse together to form a $GL$ and $GU$ factor of $G_t$.
\end{proof}
\begin{lemma}\label{lemma GU}
    If $r$ is an odd divisor of $c$ , then $GU_b(q^c)\subseteq GU_{br}(q^{c/r})$.
\end{lemma}
\begin{proof}
    Let $\sigma$ be a degree 2 automorphism on $\mathbb{F}_{q^{2c}}$ such that $\sigma(x)=x^{q^c}$. Let $h$ be a $\sigma-$ hermitian map which defines $GU_b(q^c)$. Let $c^{\prime}=c/r$.  As $r$ is odd and  $x^{q^{2c^{\prime}}}=x$ for all $x$ in $\mathbb{F}_{q^{2c^{\prime}}}$,  $\sigma(x)=x^{q^c}=x^{q^{c^{\prime}}}$. Hence $\sigma$  is an automorphism on $\mathbb{F}_{q^{2c^{\prime}}}$. To simplify the notation, let us define that $K:=\mathbb{F}_{q^{2c}} $ and $k:= \mathbb{F}_{q^{2c^{\prime}}}$. As $c^{\prime}|c$, $k$ is a subfield of $K$. Consider the map $h^{\prime}= Tr^K_{k}\circ h$.  Since $Tr_k^K$ commutes with $\sigma$, $h^{\prime}$ is $\sigma-$ hermitian on $k^{br}$. Now $GU_b(q^c)$ preserves $h^{\prime}$ due to the fact that $GU_b(q^c)$ preserves $h$. Hence, $GU_b(q^c)\subseteq GU_{br}(q^{c/r})$.
\end{proof}
Let $G$ be a pseudo-Levi subgroup containing $T_{\lambda,\rho}$ such that $G=Sp_{2k}(q)\times \prod_{i=1}^pGL_{a_i}(q^{d_i})\times \prod_{i=1}^uGU_{b_i}(q^{c_i})$ where $k+\sum_{i=1}^p a_id_i+\sum_{i=1}^u b_ic_i=n$ and each $d_i$ is a factor of some $\lambda_j$ and $c_i$ is a factor of some $\rho_j$. 
Also, note that  we know from the previous section,
\begin{align*}
    GL_{a_i}(q^{d_i})\subseteq GL_{a_im}(q^{d_i/m})...\subseteq GL_{a_id_i}(q)\subseteq Sp_{2a_id_i}(q),\text{\quad where $m$ is a divisor of $d_i$.}
\end{align*}

As we assumed that $\rho$ is a partition consisting of odd integers, from Lemma \ref{lemma GU} a similar inclusion like $GL_n(q)$ is also applicable to $GU$ factors. 
\begin{align*}
    GU_{b_i}(q^{c_i})\subseteq GU_{b_ir}(q^{c_i/r})...\subseteq GU_{b_ic_i}(q)\subseteq Sp_{2b_ic_i}(q),\text{\quad where $r$ is a divisor of $c_i$.}
\end{align*}

The $Sp$ factor will always be mentioned as the first term of the product.\\ 
In this section, we will compute the Möbius function $\mu(G_1, G_2)$, where 
\begin{align*}
    G_1=Sp_{2m}(q)\prod_{i=1}^pGL_{a_i}(q^{d_i})\prod_{i=1}^{p^{\prime}}GU_{a^{\prime}_i}(q^{d^{\prime}_i})\subseteq G_2=Sp_{2k}(q)\prod_{i=1}^uGL_{b_i}(q^{c_i})\prod_{i=1}^{u^{\prime}}GU_{b^{\prime}_i}(q^{c^{\prime}}_i)
\end{align*} 

Clearly, $G_1$ can be factored as $G^{(1)}_1$ and $G^{(2)}_1$, such that 
\begin{align*}
    G^{(1)}_1=Sp_{2m}(q)\prod_{i=1}^jGL_{a_i}(q^{d_i})\prod_{i=1}^{j^{\prime}}GU_{a^{\prime}_i}(q^{d^{\prime}_i})\subseteq Sp_{2k}(q),\text{for some $j\leq p,j^{\prime}\leq p^{\prime}$}
\end{align*}
and  $G_1^{(2)}$ is a subgroup of $\prod_{j=1}^uGL_{b_j}(q^{c_j})\prod_{j=1}^{u^{\prime}}GU_{b^{\prime}_j}(q^{c^{\prime}_j})$. As $G_1^{(2)}$ does not have any $Sp$ factor, we can determine $\mu(G_1^{(2)},\prod_{i=1}^uGL_{b_i}(q^{c_i})\prod_{i=1}^{u^{\prime}}GU_{b^{\prime}_i}(q^{c^{\prime}_i}))$ by Theorem \ref {thm for p<k is 0} and Theorem \ref{theorem 2}. Theorem \ref {thm for p<k is 0} and Theorem \ref{theorem 2} hold verbatim for products of $GU$ factors, with Lemma \ref{lemma for GU}, replacing Lemma 1 for the inclusions $GU_b(q^c)\subseteq GU_{br}(q^{c/r}).$\\
It is evident that $k\geq m$ and $[Sp_{2m}(q)\prod_{i=1}^jGL_{a_i}(q^{d_i})\prod_{i=1}^{j^{\prime}}GU_{a^{\prime}_i}(q^{d^{\prime}_i}),Sp_{2k}(q)]$ is isomorphic to \\
$[\prod_{i=1}^jGL_{a_i}(q^{d_i})\prod_{i=1}^{j^{\prime}}GU_{a^{\prime}_i}(q^{d^{\prime}_i}),Sp_{2(k-m)}(q)]$.
Hence, from now on we will assume \\
\begin{align*}
   G_1=\prod_{i=1}^pGL_{a_i}(q^{d_i})\prod_{i=1}^uGU_{b_i}(q^{c_i})\subseteq G_2=Sp_{2k}(q). 
\end{align*}

If $k>\sum_{i=1}^p a_id_i+\sum_{i=1}^u b_ic_i$ then, $[G_1,G_2]=[G_1,Sp_{2\sum a_id_i+b_ic_i}(q)]\cup[Sp_{2\sum a_id_i+b_ic_i}(q),G_2]$ and therefore, $\mu(G_1,G_2)=0$ from Lemma \ref{lemma 10(1)}. So, we assume that $k=\sum_{i=1}^p a_id_i+\sum_{i=1}^u b_ic_i$.
\begin{lemma} \label{lemma 10(1)}
    Let $P$ be a poset and $x,y$ be two elements in $P$ such that $x\leq y$. Assume there exists an element $z\in [x,y]$ such that $z\notin \{x,y\}$ and $[x,y]=[x,z]\cup[z,y]$, then for all $t\in[z,y]\setminus\{z\}, \mu(x,t)=0.$ 
\end{lemma}
\begin{proof}
       Let us define a notion of rank for this paper. Let $t$ be an element in $[z,y]$. The maximal length of a maximal chain in $[z,t]$ is called the rank of $t$. Clearly, only $z$ has rank $1$.\\
    Consider an element $t$ which is of rank $2$. That implies $[z,t]=\{z,t\}$. Now, from the definition of the  Möbius function, we know that $\mu(x, t)=-\sum_{x\leq u<t}\mu(x,u).$ But 
    \begin{align*}
      \sum_{x\leq u<t}\mu(x,u)=\sum_{x\leq u\leq z}\mu(x,u)=\sum_{x\leq u<z}\mu(x,u)+\mu(x,z)=0\\
      [\text{by the recursive definition of $\mu(x,z)$}]  
    \end{align*}
We will prove the theorem by inducting on the rank of an element. Let $t_k$ be an element in $[z,y]$ of rank $k$ and for all elements $u$ of rank strictly less than $k$, $\mu(x,u)=0$. Therefore, $\mu(x,t_k)=-\sum_{x\leq u<t_k}\mu(x,u)=-\sum_{x\leq v\leq z}\mu(x,v)-\sum_{ w\neq z: rank(w)<k} \mu(x,w)=0$. 
[The first term is $0$ by the recursive definition of $\mu(x, z)$ and the second term is zero by the induction hypothesis]. Therefore, by induction, $\mu(x,t)=0$ for all $t\in[z,y]\setminus \{z\}.$
\end{proof}  
\subsection{All exponents $d_i$ and $c_i$ are $1$}
\subsubsection{Each $d_i$ is 1 and $G_1$ has no $GU$ factor}
In this subsubsection, we assume $G_1=\prod_{i=1}^pGL_{a_i}(q)\subseteq G_2=Sp_{2k}(q)$, where $k=\sum_{i=1}^p a_i$, and determine $\mu(G_1,G_2)$.\\
\textbf{Definition of a separated partition:}\\
A \textit{separated partition} of $[p]$, is a tuple $z:=(U,V)$ where $U\subseteq \{1,\dots,p\}$ and $V$  is a partition of the set $\{1,\dots p\}\setminus U$.  Let $|A|$ be the cardinality of a set $A$. So, $V$ is an element of $\Pi_{p-|U|}$, where $\Pi_r$ is the set of all partitions of a set of cardinality $r$. We define $S_p$ to be the set of possible separated partitions of $[p]$. We denote by $n(z)$ the number of blocks present in $V$, where $z=(U,V)$.\\
\textbf{Separated partition as a poset:}\\
Let $x=(U_1,V_1),y=(U_2,V_2)$ be two elements of $S_p$. We define $x\leq y$ if $U_1\subseteq U_2$ and $V_1\setminus U_2\leq V_2$ in $\Pi_{p-|U_2|}$, where $V_1\setminus U_2$ is a partition of $\Pi_{p-|U_2|}$ obtained by deleting elements of $U_2$ from the blocks of $V_1$.
\begin{lemma}
    Assume $q\geq k+1$. The set of pseudo-Levi subgroups in $[G_1,G_2]$ is isomorphic as a poset to $S_p$, where $G_1$ is a product of $p$ many factors. 
\end{lemma}
\begin{proof}
As $q\geq k+1$, all the elements in $\{1,\dots k\}$ are non-zero in $\mathbb{F}_q$.
    Let $f:[G_1,G_2]\to S_p$ such that $f(G_1)$ is the discrete partition, the set $1|2|\dots|p$.  Assume that $G$ is an element of $[G_1,G_2]$ such that $G=Sp_{2r} \times G_3$, where $G_3$ has no $Sp$ factor. Then $f(G)$ will be $(U_G,V_G)$ where $U_G=\{i|GL_{a_i}(q)\text{transformed to the $Sp$ factor of $G$}\}$  and the $i$th block of $V_G$ equals $\{j|GL_{a_j}(q)\subseteq G_3^{(i)}\}$, where $G_3^{(i)}$ is the $ith$ factor of $G_3$. It is quite intuitive from the way we defined the partial ordering on $S_p$ , $f$ is an order-preserving bijection between $[G_1,G_2]$ and $S_p$.
\end{proof}
\begin{definition}
   We define an operation on the elements of $S_p$. \textit{Addition and colouring} is an operation on $x=(U,V)$ to move some blocks of $V$ into $U$, and to colour the remaining blocks of $V$ by $t$ many colours where $t>p$.
\end{definition}
\begin{lemma}\label{lemma for GL}
    $G_1=\prod_{i=1}^pGL_{a_i}(q)\subseteq G_2=Sp_{2k}(q)$, where $k=\sum_{i=1}^p a_i$, then $\mu(G_1,G_2)=(-1)^pp!$.
\end{lemma}
\begin{proof}
    The number of addition and colourings of $x$ is $(t+1)^{n(x)}$, as for each block there are $t+1$ many options (either to be coloured by $t$ many colours or get merged to $U$).  Our goal is to first derive $\mu(x,G_2)$ for an $x\neq G_2$ and then determine $\mu(G_1,G_2)$. As $x\neq G_2$, $n(x)\geq 1$.
Let $y\in [x,G_2]$ such that $y=(U_2,V_2)$. A distinct colouring of $y$ is a colouring of blocks of $V_2$ such that different blocks of $V_2$ receive distinct colours. Let $f(n(y))$ be the number of distinct colourings of $y$. If $n(y)\geq 1$,  then $f(n(y))=t(t-1)\dots (t-n(y)+1)$. If $n(y)=0$ then $f(0)=1$.\\
For each addition and colouring of $x=(U,V)$ if we merge the same coloured blocks remaining in $V$ we get a distinct colouring of an element $y\in [x,G_2]$. Also, for each $y\in[x,G_2]$ a distinct colouring of $y$ is an addition and colouring of $x$. Hence,
\begin{align*}
    &(t+1)^{n(x)}=\sum_{y\geq x} f(n(y))\\
    \implies &f(n(x))=\sum_{y\geq x} \mu(x,y)(t+1)^{n(y)}  \text{\quad[by definition of M\"obius function]}\\
    \implies &t(t-1)(t-2)\dots (t-n(x)+1)=\sum_{y\geq x} \mu(x,y)(t+1)^{n(y)} \text{\quad [As $n(x)\geq 1$]}.
\end{align*}
 
After replacing $t+1$ by $t^{\prime}$ and we get 
\begin{align*}
    &(t^{\prime}-1)(t^{\prime}-2)\dots (t^{\prime}-n(x))=\sum_{y\geq x} \mu(x,y)(t^{\prime})^{n(y)}
\end{align*}
The constant term of the right-hand side of the last equation is $\mu(x,G_2)$ as $n(y)=0$ if and only if $y=G_2$.  So, if we equate the constant terms of both sides of the last equation, we get $(-1)^{n(x)}(n(x))!=\mu(x,G_2)$, when $n(x)\geq 1$ , i.e., $x\neq G_2$.\\
So, $\mu(G_1,G_2)=(-1)^{p}p!$, where $G_1$ has no $Sp$ factor and it is the product of $p$ many $GL_{a_i}(q)$. \\
\end{proof}

\subsubsection{$G_1$ has both $GL$ factors and $GU$ factors}
\begin{lemma}\label{lemma for GU}    
    If $G_1=\Pi_{i=1}^uGU_{b_i}(q)$, and $G_2=Sp_{2k}$, where $k=\sum_{i=1}^u b_i$ then $\mu(G_1,G_2)=(-1)^uu!$. 
\end{lemma}
\begin{proof}
    This result can be shown by a similar way as the Lemma \ref{lemma for GL}.
\end{proof}
Before going into further details, we note that This result follows from the similar arguments used in the last section.\\
In this subsubsection, we assume $G_1=\prod_{i=1}^pGL_{a_i}(q)\prod_{i=1}^uGU_{b_i}(q)$ and $G_2$ is $Sp_{2k}(q),$ where $k=\sum_{i=1}^p a_i+\sum_{i=1}^u b_i$.\\
Let $G$ be the product of $Sp_{2r}\times G_3\times G^{\prime}_3$ where $G_3$ has no $GU$ factor and $G_3^{\prime}$ has no $GL$ factor. As we mentioned at the beginning of this section, there is no inclusion relation between $GU$ factors and $GL$ factors, each $G\in[G_1,G_2]$ can be written as a tuple $(Sp_{2r_1}\times G_3,Sp_{2r_2}\times G^{\prime}_3)$ such that $r_1+r_2=r$, $Sp_{2r_1}$ formed by the $GL$ factors, and $Sp_{2r_2}$ formed by the $GU$ factors. Hence, the poset $[G_1,G_2]$ can be expressed as product of the two poset $[ \prod_{i=1}^pGL_{a_i}(q),Sp_{2v}(q)],[\prod_{i=1}^uGU_{b_i}(q),Sp_{2w}(q) ] $, where $v=\sum a_i,w=\sum b_i$. By Proposition 3.8.2 of \cite{stanleyec1}, we know that $\mu(G_1,G_2)$ is the product of $\mu(\prod_{i=1}^p GL_{a_i}(q), Sp_{2v}(q))$ and \\
$ \mu(\prod_{i=1}^u GU_{b_i}(q), Sp_{2w}(q))$.
Using  Lemma \ref{lemma for GL} and Lemma \ref{lemma for GU},
\begin{align*}
\mu(G_1,G_2)
&= \mu\left(\prod_{i=1}^p GL_{a_i}(q), Sp_{2v}(q)\right)
   \times \mu\left(\prod_{i=1}^u GU_{b_i}(q), Sp_{2w}(q)\right)\\
&= (-1)^p p! \times (-1)^u u!\\
&= (-1)^{p+u}p!u!
\end{align*}

\subsection{Not all $d_i$ and $c_j$ are $1$}
Let us recall our assumption on $G_1$ and $G_2$: $G_1=\prod_{i=1}^pGL_{a_i}(q^{d_i})\prod_{i=1}^uGU_{b_i}(q^{c_i})$ and $G_2=Sp_{2k}(q)$, where $k=\sum_{i=1}^p a_id_i+\sum_{i=1}^u b_ic_i$ and our main goal is to determine $\mu(G_1,G_2)$.
We have already computed $\mu(G_1,G_2)$ when all $d_i$ and $c_i$ are $1$. Now, we will determine the general case. It is clear that $[G_1, G_2]$ is a lattice. So, we can use Weisner's theorem (Theorem \ref{Weisner's thm}) for the computation of $\mu(G_1,G_2)$. Let us start the computation by introducing a new notation.\\
\begin{definition}
    Let $y$ be an element of $[G_1, G_2]$. A factor of $y$ of the form $GL_n(q)$ or $GU_m(q)$ is called a \textit{unit factor} for some natural numbers $n$ and $m$, and factors of the form $GL_n(q^d)$ or $GU_m(q^d)$ where $d>1$ are called \textit{non-unit factors} of $y$. For example in $y=GL_2(q)GL_3(q^2)GU_2(q)GU_3(q^2)$, the unit factors are $GL_2(q)$ and $GU_2(q)$ and the non-unit factors are $GL_3(q^2)$ and $GU_3(q^2)$. 
 \end{definition}
\begin{proposition}\label{prop 2.9}
    Assume that $G_1=\prod_{i=1}^pGL_{a_i}(q^{d_i})\prod_{i=1}^uGU_{b_i}(q^{c_i})$ and $G_2=Sp_{2k}(q)$, where $k=\sum_{i=1}^p a_id_i+\sum_{i=1}^u b_ic_i$ and $G_1$ has no unit factor. Let $x$ be an element in $[G_1,G_2]$ such that $x\wedge a=G_1$, where $a=GL_w(q)GU_v(q)$, and $w=\sum_{i=1}^p a_id_i,v=\sum_{i=1}^u b_ic_i$. Then $x=G_1$.
\end{proposition} 
\begin{proof}
     \textbf{Case 1}. $x$ has no $Sp$ factor.\\
     If $x$ has no $Sp$ factor, then $x\leq a$ and so, $x\wedge a=x$ , which means $x=G_1$.\\
     \textbf{Case 2}. $x$ has an $Sp$ factor.\\
     Let $x=Sp_{2r}(q)\prod GL_{e_i}(q^{f_i})\prod GU_{g_i}(q^{h_i})$ in $[G_1,G_2]$. The $Sp$ factor is formed by absorbing factors of the form $GL_{r_i}(q)$ and $GU_{r_j}(q)$.
     So, there is an element $z\leq x$ which doesn't have any $Sp$ factor and has unit factors that form the $Sp$ factor of $x$. Clearly, $z\leq a$ and $z$ differ from $G_1$, since $G_1$ has no unit factor. So, $x\wedge a=z\neq G_1$.\\
     Therefore, $x\wedge a=G_1$ if and only if $x=G_1$.
 \end{proof}
 \begin{remark} \label{remark for d_i>1}
     If all $d_i,c_j>1$ then, Proposition \ref{prop 2.9} and Weisner's theorem (Theorem \ref{Weisner's thm}) together imply that $\mu(G_1,G_2)=0$.
 \end{remark}

Now, let us assume that all $d_i,c_j\geq 1$ but not all of them are $1$. The case where all $c_i$ and $d_j$ are $1$ is already done in the previous subsection.\\
Assume that $a=GL_w(q)GU_v(q)$, where $w=\sum_{i=1}^p a_id_i,v=\sum_{i=1}^u b_ic_i$. Before going to further details, let us state some more definitions. 
 \begin{definition}
     The \textit{length} of an element $x\in [G_1,G_2]$ is the number of unit factors present in $x$. For example, the length of $GL_2(q)GL_3(q^2)GU_2(q)GU_3(q^2)$ is $2$ as it has two unit factors, $GL_2(q)$ and $GU_2(q)$.
 \end{definition}
 
\begin{definition}
    Let $z$ be an element in $[G_1,G_2]$. The \textit{dimension} of a factor of $z$ of the form $GL_m(q^n)$ or $GU_m(q^n)$ is defined as the integer $mn$. The \textit{dimension of an $Sp$ factor} of the form $Sp_{2r}(q)$ is $r$.\\
    The \textit{unit dimension} of $z$ is defined as the sum of the dimensions of the unit factors of $z$. Similarly, the \textit{non-unit dimension} of $z$ is the sum of the dimensions of the non-unit factors of $z$. The \textit{$Sp$ rank} of $z$ is the dimension of the $Sp$ factor of $z$.
\end{definition}

\begin{definition}
    For any $z$ in $[G_1, G_2]$, the sum of unit dimension, non-unit dimension and $Sp$ rank is called the \textit{dimension} of $z$.
\end{definition}
It is clear that for any $z$ in $[G_1,G_2]$, the dimension of $z$ is $k$, where $G_2=Sp_{2k}(q)$ and $G_1=\prod_{i=1}^pGL_{a_i}(q^{d_i)}\prod_{i=1}^uGU_{b_i}(q^{c_i})$ and $k=\sum_{i=1}^p a_id_i+\sum_{i=1}^u b_ic_i$. 

\begin{lemma} \label{lemma 2.11}
    Recall that  $G_1=\prod_{i=1}^pGL_{a_i}(q^{d_i})\prod_{i=1}^uGU_{b_i}(q^{c_i})$ and $a=GL_w(q)GU_v(q)$, where $w=\sum a_id_i,v=\sum b_ic_i$ and $G_2=Sp_{2k}(q),k=w+v$. Then, if $x\wedge a=G_1$, then the non-unit dimension of $x$ is the same as that of $G_1$.
\end{lemma}
\begin{proof}
    As $x\geq G_1$ and the dimensions of $x$ and $G_1$ are the same, the non-unit dimension of $x$ will be less than or equal to that of $G_1$. If the non-unit dimension of $x$ is smaller than the non-unit dimension of $G_1$, then there exists a minimal element (w.r.t containment) $y$ in $[ G_1,x]$ such that the non-unit dimension of $y$ is smaller than that of $G_1$. If $y$ has an $Sp$ factor, then there exists an element $y_1\leq y$ such that $y_1$ has no $Sp$ factor (by unwinding the $Sp$ factors of $y$, we can construct $y_1$). Clearly, the non-unit dimensions of $y$ and $y_1$ are the same as the ranks of the $Sp$ factors formed by absorbing unit elements, which don't contribute to the non-unit dimension. But we assume that $y$ is minimal
among all the elements whose non-unit dimension is smaller than that of $G_1$. Hence, $y$ cannot have any $Sp$ factor. So, $y\leq a$  which implies $x\wedge a\geq y> G_1$. That contradicts the assumption of the lemma. So, the non-unit dimension of $x$ must be the same as the non-unit dimension of $G_1$.
\end{proof}
\begin{proposition}\label{prop 2.11}
    Let $x$ be an element in $[G_1, G_2]$ such that $x\wedge a=G_1$, where $a=GL_w(q)GU_v(q)$ (defined in Lemma \ref{lemma 2.11}). If $x$ has a non-zero $Sp$ rank, then the length of $x$ is strictly less than that of $G_1$.
\end{proposition}
\begin{proof}
    Recall that the length of $x$ is the number of unit factors present in $x$. The non-unit dimension of $x$ and $G_1$ is the same [from Lemma \ref{lemma 2.11}], so the sum of the $Sp$ rank and unit dimension of $x$ and $G_1$ must be the same.  Also note that the $Sp$ rank of $G_1$ is $0$, and the $Sp$ rank of $x$ is positive. So, the unit dimension of $x$ must be smaller than that of $G_1$. As $x\geq G_1$, the unit factors cannot be split into other factors, which leads to a decrease in the number of unit factors. So, the number of unit factors of $x$, which is the length of $x$, must be smaller than the length of $G_1$.
\end{proof}
\begin{theorem}\label{main thm for SP}   
    Let $G_1=\prod_{i=1}^{p}GL_{a_i}(q^{d_i})\prod_{i=1}^uGU_{b_i}(q^{c_i)}$ where $c_i,d_i\geq 1$ and \\ $G_2=Sp_{2(\sum_{i=1}^p a_id_i+\sum_{i=1}^u b_ic_i)}(q)$. If $G_1$ has at least one non-unit factor, then $\mu(G_1, G_2)=0$.
\end{theorem}
\begin{proof}
   We will use induction on the length of $G_1$. If the length of $G_1$ is $0$ i.e., all $d_i,c_j>1$ then we have already proved that $\mu(G_1,G_2)=0$ in Remark \ref{remark for d_i>1}.\\
   Let the length of $G_1$ be $l>0$. By the induction hypothesis, we assume the statement holds when the length of $G_1$ is strictly less than $l$. From Weisner's theorem (Theorem \ref{Weisner's thm}) we get that $\sum_{x:x\wedge a=G_1}\mu(x,G_2)=0$, where $a=GL_w(q)GU_v(q)$, $w=\sum_{i=1}^p a_id_i$ and $v=\sum_{i=1}^u b_ic_i $. By Proposition \ref{prop 2.11}, we know that if $x\wedge a=G_1$ then either $x$ has no $Sp$ factor or the length of $x$ is strictly smaller than $l$. By Lemma \ref{lemma 2.11}, we get if $x$ has no $Sp$ factor, then $x\leq a$, which means $x=G_1$. If $x$ has an $Sp$ factor, then the length of $x$ is strictly less than $l$. Therefore, by the induction hypothesis, for each $x\neq G_1$ satisfying $x\wedge a=G_1$, $\mu(x, G_2)=0$. Hence, by Weisner's theorem, we conclude that $\mu(G_1, G_2)=0$.
\end{proof}

\bibliographystyle{plain}
\bibliography{references}

\end{document}